\documentclass[11pt,a4paper]{amsart}
\usepackage[english]{babel}
\usepackage{amsmath,amsthm,amssymb,amsfonts}
\usepackage[pdftex]{color}
\usepackage[bookmarks=true,hyperindex,pdftex,colorlinks,citecolor=blue,
linkcolor=blue,urlcolor=blue]{hyperref}

\usepackage{graphicx}
\usepackage[shortlabels]{enumitem}
\usepackage{mathtools}
\usepackage{mathabx}

\makeatletter
\@namedef{subjclassname@2020}{%
  \textup{2020} Mathematics Subject Classification}
\makeatother

\theoremstyle{plain}
\newtheorem{theorem}{Theorem}[section]
\newtheorem{cor}[theorem]{Corollary}
\newtheorem{prop}[theorem]{Proposition}
\newtheorem{lemma}[theorem]{Lemma}

\theoremstyle{definition}

\newtheorem{examples}[theorem]{Examples}
\newtheorem{definition}[theorem]{Definition}
\newtheorem{remark}[theorem]{Remark}

\DeclareMathOperator{\supp}{supp}
\DeclareMathOperator{\na}{NA}
\DeclareMathOperator{\nra}{NRA}

\title[On the BPBp-nu for compact operators]{On the compact operators case of the Bishop--Phelps--Bollob\'{a}s property for numerical radius}

\author[D.\ Garc\'{\i}a]{Domingo Garc\'{\i}a}
\address[Domingo Garc\'{\i}a]{Departamento de An\'{a}lisis Matem\'{a}tico,
Universidad de Valencia, Doctor Moliner 50, 46100 Burjasot (Valencia), Spain
\href{http://orcid.org/0000-0001-5291-6705}{ORCID: \texttt{0000-0002-2193-3497} }}
\email{domingo.garcia@uv.es}

\author[M.\ Maestre]{Manuel Maestre}
\address[Manuel Maestre]{Departamento de An\'{a}lisis Matem\'{a}tico,
Universidad de Valencia, Doctor Moliner 50, 46100 Burjasot
(Valencia), Spain
\href{http://orcid.org/0000-0002-2193-3497}{ORCID: \texttt{0000-0001-5291-6705} }}
\email{manuel.maestre@uv.es}

\author[M.\ Mart\'{\i}n]{Miguel Mart\'{\i}n}
\address[Miguel Mart\'{\i}n]{Departamento de An\'{a}lisis Matem\'{a}tico, Facultad de
 Ciencias, Universidad de Granada, 18071 Granada, Spain.
\href{http://orcid.org/0000-0003-4502-798X}{ORCID: \texttt{0000-0003-4502-798X} }}
\email{mmartins@ugr.es}

\author[O.\ Rold\'an]{\'Oscar Rold\'an}
\address[\'Oscar Rold\'an]{Departamento de An\'{a}lisis Matem\'{a}tico,
Universidad de Valencia, Doctor Moliner 50, 46100 Burjasot (Valencia), Spain.
\href{https://orcid.org/0000-0002-1966-1330}{ORCID: \texttt{0000-0002-1966-1330} }}
\email{oscar.roldan@uv.es}

\thanks{The first and second authors were supported by MINECO and FEDER project MTM2017-83262-C2-1-P and by Prometeo PROMETEO/2017/102.  The third author was supported by projects PGC2018-093794-B-I00 (MCIU/AEI/FEDER, UE), A-FQM-484-UGR18 (Universidad de Granada and Junta de Analuc\'{\i}a/FEDER, UE), and FQM-185 (Junta de Andaluc\'{\i}a/FEDER, UE). The fourth author was supported by the Spanish Ministerio de Ciencia, Innovaci\'on y Universidades, grant FPU17/02023, and by  MINECO and FEDER project MTM2017-83262-C2-1-P}


\date{February 21, 2021}


\begin{document}
	
\begin{abstract}
We study the Bishop-Phelps-Bollob\'as property for numerical radius restricted to the case of compact operators (BPBp-nu for compact operators in short). We show that $C_0(L)$ spaces have the BPBp-nu for compact operators for every Hausdorff topological locally compact space $L$. To this end, on the one hand, we provide some techniques allowing to pass the BPBp-nu for compact operators from subspaces to the whole space and, on the other hand, we prove some strong approximation property of $C_0(L)$ spaces and their duals. Besides, we also show that real Hilbert spaces and isometric preduals of $\ell_1$ have the BPBp-nu for compact operators.
\end{abstract}

\maketitle
\section{Introduction, notation, and known results}\label{SectionIntroduction}

First we fix some notation in order to be able to describe our aims and results with precision. Given a Banach space $X$ over the field $\mathbb{K}$ of real or complex numbers, we denote by $X^*$, $B_X$, and $S_X$, its topological dual, its closed unit ball, and its unit sphere, respectively. If $Y$ is another Banach space, $\mathcal{L}(X,Y)$ represents the space of all bounded and linear operators from $X$ to $Y$, and we denote by $\mathcal{K}(X, Y)$ the space of compact operators from $X$ to $Y$. When $Y=X$, we shall simply write $\mathcal{L}(X)=\mathcal{L}(X,X)$ and $\mathcal{K}(X)=\mathcal{K}(X,X)$. Given a locally compact Hausdorff topological space $L$, $C_{0}(L)$ is the Banach space of all scalar-valued continuous functions on $L$ vanishing at infinity.

Given an operator $T\in \mathcal{L}(X)$, its \emph{numerical radius} is defined as
$$
\nu(T):= \sup \left\{ |x^*(T(x))|\colon  (x,\, x^*)\in \Pi(X) \right\},
$$
where $\Pi(X):= \left\{ (x,\, x^*)\in S_X \times S_{X^*}\colon  x^*(x)=1 \right\}$. It is immediate that $\nu(T)\leq \| T \|$ for every $T\in \mathcal{L}(X)$ and that $\nu$ is a seminorm on $\mathcal{L}(X)$. Very often, $\nu$ is actually a norm on $\mathcal{L}(X)$ equivalent to the usual operator norm. The \emph{numerical index} of the space $X$ measures this fact and it is given by
\begin{align*}
 n(X)&:=   \inf \{ \nu(T)\colon T\in \mathcal{L}(X),\, \| T\| = 1\} \\
   & \,= \max \{ k\geq 0\colon k\| T \| \leq \nu(T),\, \forall T\in  \mathcal{L}(X)\}.
\end{align*}
It is clear that $0\leq n(X)\leq 1$ and that $n(X)>0$ if and only if the numerical radius is a norm on $\mathcal{L}(X)$ equivalent to the operator norm. As in this paper we will mainly deal with compact operators, we will also need the following concept from \cite{CMM}. Given a Banach space $X$, the \emph{compact numerical index} of $X$ is
\begin{align*}
n_K(X)&:= \inf \{ \nu(T)\colon T\in \mathcal{K}(X),\, \| T\| = 1\} \\
&= \max \{ k\geq 0\colon k\| T \| \leq \nu(T),\, \forall T\in \mathcal{K}(X)\}.
\end{align*}
We refer the reader to \cite{CMM}, \cite{KaMaPa}, \cite{KaMaMeQue-diss}, \cite[Subsection~1.1]{KLMM}, and references therein for more information and background.

An operator $S\in\mathcal{L}(X,\, Y)$ is said to \emph{attain its norm} whenever there exists some $x\in S_X$ such that $\| S \| = \| S(x) \|$. An operator $T\in \mathcal{L}(X)$ is said to \emph{attain its numerical radius} whenever there exists some $(x,\, x^*)\in \Pi(X)$ such that $\nu(T)=|x^*(T(x))|$. The sets of norm attaining operators from $X$ to $Y$ and of numerical radius attaining operators on $X$ will be denoted, respectively, by $\na(X,\, Y)$ and $\nra(X)$.

In 1961, Bishop and Phelps \cite{BP} proved that the set $\na(X,\mathbb{K})$ of norm attaining functionals on a Banach space $X$ is always dense in $X^*$. However, this result has been shown to fail for general operators between Banach spaces, as Lindenstrauss \cite{L} proved in 1963. We refer the reader to the survey \cite{Ac2006} for more information and background on the density of norm attaining operators, and to \cite{Ma2016} for the compact operators version.

In 1970, Bollob\'{a}s \cite{B} gave a refinement of the Bishop-Phelps Theorem, showing that you can approximate simultaneously a functional and a point where it almost attains its norm by a norm-attaining funcional and a point where the new functional attains its norm, respectively. In order to extend Bollob\'{a}s' result to norm attaining operators between Banach spaces, Acosta, Aron, Garc\'{\i}a and Maestre \cite{AAGM} introduced in 2008 the Bishop-Phelps-Bollob\'as property as follows.

\begin{definition}[\textrm{\cite{AAGM}}]
A pair of Banach spaces $(X, Y)$ has the \emph{Bishop-Phelps-Bollob\'as property} (\emph{BPBp} for short) if given $\varepsilon\in (0,1)$, there exists $\eta(\varepsilon)\in (0,1)$ such that whenever $T\in \mathcal{L}(X,\, Y)$ and $x_0\in S_X$ satisfy $\| T\| = 1$ and $\| T(x_0)\| > 1-\eta(\varepsilon)$, there are $S\in \mathcal{L}(X,\, Y)$ and $x_1\in S_X$ such that
$$\| S \| = \| S(x_0)\| = 1,\quad \|x_0 - x_1\| < \varepsilon,\quad \| S - T\| < \varepsilon.$$
If the above property holds when we restrict the operators $T$ and $S$ to be compact, we say that the pair $(X,Y)$ has the \emph{Bishop-Phelps-Bollob\'{a}s property for compact operators} (\emph{BPBp for compact operators} for short).
\end{definition}

With the above notation, the result by Bollob\'{a}s just says that the pair $(X,\mathbb{K})$ has the Bishop-Phelps-Bollob\'{a}s property for every Banach space $X$. In the paper \cite{AAGM} a variety of pairs of spaces satisfying the BPBp are provided, together with examples of pairs $(X,Y)$ of Banach spaces failing the BPBp for which $\na(X,Y)$ is dense in $\mathcal{L}(X,Y)$. We refer the reader to the survey \cite{Ac2019} and the paper \cite{ACKLM} for more information and background on the BPBp.

Motivated by this property, Guirao and Kozhushkina \cite{GK} introduced in 2013 the Bishop-Phelps-Bollob\'as property for numerical radius as follows.

\begin{definition}[\textrm{\cite{GK}}]
A Banach space $X$ is said to have the \emph{Bishop-Phelps-Bollob\'as property for numerical radius} (\emph{BPBp-nu} for short) if for every $0<\varepsilon < 1$, there exists $\eta(\varepsilon)\in (0,1)$ such that whenever $T\in \mathcal{L}(X)$ and $(x,\, x^*)\in \Pi(X)$ satisfy $\nu(T)=1$ and $|x^*(T(x))|>1-\eta(\varepsilon)$, there exist $S\in \mathcal{L}(X)$ and $(y,\, y^*)\in \Pi(X)$ such that
$$\nu(S)=|y^*(S(y))|=1,\quad \| T-S\|<\varepsilon,\quad \| x-y\|<\varepsilon,\quad \| x^*-y^*\|<\varepsilon.$$
\end{definition}

Since then, several works have been done in order to study what spaces satisfy that property. We summarize next some of the most important results on the matter:
\begin{enumerate}
\item The spaces $c_0$ and $\ell_1$ have the BPBp-nu \cite{GK}.
\item $L_1(\mathbb{R})$ has the BPBp-nu \cite{F}.
\item Finite-dimensional spaces have the BPBp-nu \cite{KLM14}.\item The real or complex space $L_p(\mu)$ has the BPBp-nu for every measure $\mu$ when $1\leq p<\infty$ (\cite[Example 8]{KLM14} except for the real case $p=2$, which is covered in \cite[Corollary 3.3]{KLMM}).
\item Any uniformly convex and uniformly smooth Banach space $X$ with $n(X)>0$ has the BPBp-nu \cite{KLM14}.
\item Every separable infinite-dimensional Banach space can be renormed to fail the BPBp-nu \cite{KLM14}, even though the set of numerical radius attaining operators is always dense in spaces with the Radon-Nikod\'ym property.
\item The real space $C(K)$ has the BPBp-nu under some extra conditions on the compact Hausdorff space $K$ (for example, when $K$ is metrizable) \cite{AGR}. Let us comment that it remains unknown if the result is true for all compact Hausdorff spaces or what happens in the complex case.
\end{enumerate}

We refer the interested reader to the cited papers \cite{AGR, F, GK, KLM14, KLMM} and the papers \cite{AFSM, CDJM} and references therein for more information and background.

In 2018, Dantas, Garc\'ia, Maestre and Mart\'in \cite{DGMM} studied the BPBp for compact operators. They presented some abstract techniques (based on results about norm attaining compact operators by Johnson and Wolfe \cite{JW}) which allow to carry the BPBp for compact operators from sequence spaces (such as $c_0$ and $\ell_p$) to function spaces (as $C_0(L)$ and $L_p(\mu)$). As one of the main results, it is shown in \cite{DGMM} that the BPBp for compact operators of the pair $(c_0,Y)$ is sufficient to get the BPBp for compact operators of all the pairs  $(C_0(L),Y)$ regardless of the locally compact Hausdorff topological space $L$.

Our aim in this paper is to study the following property, inspired both by the BPBp for compact operators and by the BPBp for numerical radius.

\begin{definition}\label{definition_BPBp-nu-K}
A Banach space $X$ is said to have the \emph{BPBp-nu for compact operators} if for every $0<\varepsilon < 1$, there exists $\eta(\varepsilon)\in (0,1)$ such that whenever $T\in \mathcal{K}(X)$ and $(x,\, x^*)\in \Pi(X)$ satisfy $\nu(T)=1$ and $|x^*(T(x))|>1-\eta(\varepsilon)$, there exist $S\in \mathcal{K}(X)$ and $(y,\, y^*)\in \Pi(X)$ such that
$$\nu(S)=|y^*(S(y))|=1,\quad \| T-S\|<\varepsilon,\quad \| x-y\|<\varepsilon,\quad \| x^*-y^*\|<\varepsilon.$$
\end{definition}

The first work where a somewhat similar property was introduced is \cite{AFSM}, where the BPBp-nu for subspaces of $\mathcal{L}(X)$ was defined and studied in the case of $L_1(\mu)$, with $\mu$ a finite measure. Let us provide a list of spaces that are known to have the BPBp-nu for compact operators.

\begin{examples}\label{EjemplosIniciales}
The following spaces have the BPBp-nu for compact operators:
\begin{enumerate}[(a)]
\item Finite dimensional spaces \cite[Proposition 2]{KLM14}.
\item $c_0$ and $\ell_1$ (adapting the proofs given in \cite[Corollaries 3.3 and 4.2]{GK}).
\item $L_1(\mu)$ for every measure $\mu$ (using \cite[Corollary 2.1]{AFSM} for finite measures and adapting \cite[Theorem 9]{KLM14} to compact operators for the general case).
\end{enumerate}
\end{examples}


Adapting the results from \cite{KLM14} and \cite{KLMM}, one also has that the $L_p(\mu)$ spaces have the BPBp-nu for compact operators when $1<p<\infty$. However, the adaptation to the compact operators case of the proofs in \cite{KLM14} and \cite{KLMM} needs to introduce some terminology. Therefore, we enounce the result here but we pospone the proof to Subsection~\ref{Subsec:LpHilbert}.

\begin{prop}\label{prop:LpProp}
$L_p(\mu)$ has the BPBp-nu for compact operators, for every measure $\mu$ and  $1<p<\infty$.
\end{prop}

Our main objective in this paper is to prove the following result, which is not covered by Examples~\ref{EjemplosIniciales}.

\begin{theorem}\label{TheoremMain}
If $L$ is a locally compact Hausdorff space, then $C_0(L)$ has the BPBp-nu for compact operators.
\end{theorem}

As a consequence, we get that $L_\infty(\mu)$ spaces have the BPBp-nu for compact operators, completing Example~\ref{EjemplosIniciales}.c and Proposition \ref{prop:LpProp}.

\begin{cor}\label{cor:LinftyProp}
$L_\infty(\mu)$ has the BPBp-nu for compact operators for every measure $\mu$.
\end{cor}

Let us recall that it is shown in \cite{AGR} that the real space $C(K)$ has the BPBp-nu for some compact Hausdorff spaces $K$, but the general case, as well as the complex case, remain open as far as we know. However, Theorem \ref{TheoremMain} gives a complete answer in the case of compact operators.

To get the proof of Theorem~\ref{TheoremMain}, we need two kind of ingredients. On the one hand, we provide in Section~\ref{SectionTools} some abstract results that will allow us to carry the BPBp-nu for compact operators from sequence spaces into function spaces, in some cases. The more general result of this kind is Lemma~\ref{LemaBPBpnuK1}, which will be the first ingredient for the proof of Theorem~\ref{TheoremMain}. It is somehow an extension of \cite[Lemma 2.1]{DGMM} but it needs more restrictive hypothesis in order to deal with the numerical radius instead of with the norm of the operators. We also provide in that section some useful particular cases of Lemma~\ref{LemaBPBpnuK1} which allow to show, for instance, that every predual of $\ell_1$ has the BPBp-nu for compact operators, see Corollary~\ref{Cor:PredualEll1}. The second ingredient for the proof of Theorem~\ref{TheoremMain} is some strong approximation property of $C_0(L)$ spaces and their duals which will be provided in Section~\ref{SectionCL} (see Theorem \ref{theorem:AP_C0(L)}) and which will allow us to apply Lemma~\ref{LemaBPBpnuK1} in this case, thus providing the proof of Theorem~\ref{TheoremMain}. Let us also comment that Theorem \ref{theorem:AP_C0(L)} gives a much stronger approximation property of $C_0(L)$ and its dual space than \cite[Lemma 3.4]{DGMM}.

\subsection{\texorpdfstring{$\boldsymbol{L_p}$}{Lp}
spaces have the BPBp-nu for compact operators}\label{Subsec:LpHilbert}\hfill

In this subsection, we will adapt the concepts and results from \cite{KLM14} and \cite{KLMM} to the compact operators case to show that $L_p(\mu)$ spaces have the BPBp-nu for compact operators for $1<p<\infty$, providing thus a proof of Proposition~\ref{prop:LpProp}.

In \cite[Definition 5]{KLM14} a weaker version of the BPBp-nu, the weak BPBp-nu, was introduced and studied. We present here the compact operators version of that property.

\begin{definition}
A Banach space $X$ is said to have the \emph{weak BPBp-nu for compact operators} if for every $0<\varepsilon < 1$, there exists $\eta(\varepsilon)\in (0,1)$ such that whenever $T\in \mathcal{K}(X)$ and $(x,\, x^*)\in \Pi(X)$ satisfy $\nu(T)=1$ and $|x^*(T(x))|>1-\eta(\varepsilon)$, there exist $S\in \mathcal{K}(X)$ and $(y,\, y^*)\in \Pi(X)$ such that
$$\nu(S)=|y^*(S(y))|,\quad \| T-S\|<\varepsilon,\quad \| x-y\|<\varepsilon,\quad \| x^*-y^*\|<\varepsilon.$$
\end{definition}

Note that this is a similar property to the BPBp-nu for compact operators, but without asking for the condition $\nu(S)=1$ (see Definition~\ref{definition_BPBp-nu-K}).

It is shown in \cite[Proposition 4]{KLM14} that uniformly convex and uniformly smooth Banach spaces have the weak BPBp-nu. This result also holds for the compact operators version by an easy adaptation of the proof.

\begin{prop}\label{prop:UCUS-wBPBpnuK}
If a Banach space is uniformly convex and uniformly smooth, then it has the weak BPBp-nu for compact operators.
\end{prop}

\begin{proof}
We can follow the proof of \cite[Proposition 4]{KLM14}, just keeping in mind that if the original operator $T_0$ is compact, then the rest of operators $T_n$ from that proof are also compact, and so, $S$ is compact too.
\end{proof}

Later, in \cite[Proposition 6]{KLM14}, it is proven that in Banach spaces with positive numerical index, the BPBp-nu and the weak BPBp-nu are equivalent. This property is also true for the compact operators versions of the properties if we use the compact numerical index.

\begin{prop}\label{prop:nXpos-wEquiv}
Let $X$ be a Banach space such that $n_K(X)>0$. Then $X$ has the BPBp-nu if, and only if, it has the weak BPBp-nu.
\end{prop}

\begin{proof}
It suffices to follow the proof from \cite[Proposition 6]{KLM14} but with both $T$ and $S$ being now compact operators, and using $n_K(X)$ instead of $n(X)$.
\end{proof}

As a consequence of these two results, similarly to what is done in \cite{KLM14}, we get that all $L_p(\mu)$ spaces have the BPBp-nu for compact operators when $1<p<\infty$ in the complex case and when $1<p<\infty$, $p\neq 2$ in the real case. This is so because, on the one hand, in the real case,
$$
n_K(L_p(\mu))\geq n(L_p(\mu))>0 \qquad \bigl(1<p<\infty,\ p\neq 2\bigr)
$$
by \cite{MarMerPop} and, on the other hand, $n_K(X)\geq 1/\mathrm{e}>0$ for every complex Banach space (see \cite[Eq.~(1) in p.\ 156]{KaMaPa}, for instance).

This provides the proof of Proposition~\ref{prop:LpProp} for all values of $p$ in $(1,+\infty)$ in the complex case and for all values of $p$ in $(1,+\infty)$ except for $p=2$ in the real case.

Our next aim is to show that real Hilbert spaces also have the BPBp-nu for compact operators, by adapting the ideas from \cite{KLMM}.

First, given a real Banach space $X$, we consider the following subset of $\mathcal{K}(X)$:
$$
\mathcal{Z}_K(X):=\bigl\{T\in \mathcal{K}(X)\colon \nu(T)=0\bigr\}
$$
which is the set of all skew-hermitian compact operators on $X$. Observe that
$$
\mathcal{Z}_K(X)=\mathcal{K}(X)\cap \mathcal{Z}(X),
$$
where $\mathcal{Z}(X)$ is the Lie-algebra of all skew-hermitian operators on $X$ (see \cite[p.~1004]{KLMM} for instance). Adapting the concept of second numerical index given in \cite{KLMM}, we define the  \emph{second numerical index for compact operators} of a Banach space $X$ as the constant
\begin{align*}
n'_K(X)&:=\inf \bigl\{\nu(T)\colon T\in \mathcal{K}(X), \| T+\mathcal{Z}_K(X)\|=1\bigr\}\\
&= \max \bigl\{ M\geq 0\colon M \| T+\mathcal{Z}_K(X)\| \leq \nu(T) \text{ for all } T\in \mathcal{K}(X)\bigr\},
\end{align*}
where $\| T+\mathcal{Z}_K(X)\|$ is the quotient norm in $\mathcal{K}(X)/\mathcal{Z}_K(X)$.

The next result is a version for compact operators of \cite[Theorem 3.2]{KLMM}.

\begin{prop}\label{prop:nK'pos-Equiv}
Let $X$ be a real Banach space with $n'_K(X)>0$. Then, the BPBp-nu for compact operators and the weak BPBp-nu for compact operators are equivalent in $X$.
\end{prop}

\begin{proof}
It suffices to adapt the steps from the proof of \cite[Theorem 3.2]{KLMM} to the case of compact operators. That is: all the involved operators $T$, $S$, $S_1$ and $S_2$ are now compact, the $\mathcal{Z}(X)$ set is replaced by $\mathcal{Z}_K(X)$, and the index $n'(X)$ is replaced by $n'_K(X)$.
\end{proof}

We are going to see next that the second numerical index for compact operators of a real Hilbert space equals one.

\begin{prop}\label{prop:n'KHpos}
Let $H$ be a real Hilbert space. Then, $n'_K(H)=1$.
\end{prop}

The proof of this result will be an adaptation of the one of \cite[Theorem 2.3]{KLMM}. Recall that in a real Hilbert space endowed with an inner product $(\cdot | \cdot)$, $H^*$ identifies with $H$ by the isometric isomorphism $x\longmapsto (\cdot | x)$. Therefore, $\Pi(X)=\{ (x, x)\in H\times H:\, x\in S_H\}$, and so, for every $T\in \mathcal{L}(H)$, one has $\nu(T)=\sup \{|(Tx|x)|\colon x\in S_H\}$. We first need to give the compact operators version of \cite[Lemma 2.4]{KLMM} whose proof is an obvious adaptation of the proof of that result.

\begin{lemma}\label{lema:HSprev}
Let $H$ be a real Hilbert space.
\begin{enumerate}[(a)]
\item $\mathcal{Z}_K(H)=\{T\in \mathcal{K}(H)\colon T=-T^*\}$.
\item If $T\in \mathcal{K}(H)$ is selfadjoint (i.e.\ $T=T^*$), then $\|T\| = \nu(T)$.
\end{enumerate}
\end{lemma}

We are now ready to present the pending proof of Proposition~\ref{prop:n'KHpos}.

\begin{proof}[Proof of Proposition \ref{prop:n'KHpos}]
It suffices to adapt the proof of \cite[Theorem 2.3]{KLMM} to the compact operators case, that is: the involved operators $T$ and $S$ are now compact, and the set $\mathcal{Z}(X)$ is replaced by $\mathcal{Z}_K(X)$.
\end{proof}

As a consequence of Propositions \ref{prop:UCUS-wBPBpnuK}, \ref{prop:nK'pos-Equiv}, and \ref{prop:n'KHpos}, we get the following result which provides the proof of the pending part of Proposition~\ref{prop:LpProp}.

\begin{cor}\label{cor:HilbertProp}
If $H$ is a real Hilbert space, then it has the BPBp-nu for compact operators.
\end{cor}

\section[First ingredient]{First ingredient: the tools}\label{SectionTools}

In this section, we will provide an abstract result that will allow us later to carry the BPBp-nu for compact operators from some sequence spaces to function spaces. The most general version that we are able to prove is the following, which is inspired in \cite[Lemma~2.1]{DGMM}, but it needs more requirements. We need some notation first. An \emph{absolute norm} $|\cdot |_a$ is a norm  in $\mathbb{R}^2$ such that $|(1, 0)|_a=|(0,1)|_a=1$ and $|(s, t)|_a=|(|s|, |t|)|_a$ for every $(s,t)\in\mathbb{R}^2$.
Given a Banach space $X$, we say that a projection $P$ on $X$ is an \emph{absolute projection} if there is an  absolute norm $|\cdot|_a$ such that $\|x\|=\bigl|(\|P(x)\|,\|x-P(x)\|)\bigr|_a$ for every $x\in X$. Examples of absolute projections are the $M$- and $L$-projections and, more in general, the $\ell_p$-projections. We refer the reader to \cite{DGMM} for the use of absolute norms with the Bishop-Phelps-Bollob\'{a}s type properties and to the references therein for more information on absolute norms.

\begin{lemma}\label{LemaBPBpnuK1}
Let $X$ be a Banach space satisfying that $n_K(X)>0$. Suppose that there is a mapping $\eta\colon (0,1)\longrightarrow (0,1)$ such that given $\delta>0$, $x_1^*,\, \ldots,\, x_n^*\in B_{X^*}$ and $x_1,\, \ldots,\, x_\ell \in B_X$, we can find norm one operators $\widetilde{P}\colon X\longrightarrow \widetilde{P}(X)$, $i\colon \widetilde{P}(X)\longrightarrow X$ such that for $P:=i\circ \widetilde{P}\colon  X\longrightarrow X$, the following conditions are satisfied:
\begin{enumerate}
\item[(i)] $\| P^* (x_j^*) - x_j^*\| < \delta$, for $j=1,\, \ldots,\, n$.
\item[(ii)] $\| P (x_j) - x_j\| < \delta$, for $j=1,\, \ldots,\, \ell$.
\item[(iii)] $\widetilde{P}\circ i = \mathrm{Id}_{\widetilde{P}(X)}$.
\item[(iv)] $\widetilde{P}(X)$ satisfies the Bishop-Phelps-Bollob\'as property for numerical radius for compact operators with the mapping $\eta$.
\item[(v)] Either $P$ is an absolute projection and $i$ is the natural inclusion, or $n_K(\widetilde{P}(X)) = n_K(X)=1$.
\end{enumerate}
Then, $X$ satisfies the BPBp-nu for compact operators.
\end{lemma}

Let us comment on the differences between the lemma above and \cite[Lemma 2.1]{DGMM}. First, condition (ii) is more restrictive here than in that lemma, where it only dealt with one point. Second, the requirements of item (v) on the compact numerical index or on the absoluteness of the projections did not appear in \cite[Lemma 2.1]{DGMM}, but they are needed here as numerical radius does not behave well in general with respect to extensions of operators.

\begin{proof}
Given $\varepsilon \in (0,\, 1)$, let $\varepsilon_0(\varepsilon)$ be the unique number with $0< \varepsilon_0(\varepsilon) < 1$ such that
$$
\varepsilon_0(\varepsilon)\left( \frac{2}{3}+\frac{1}{(1-\varepsilon_0(\varepsilon))\,n_K(X)} \right)= \varepsilon,
$$
which, in particular, satisfies that $\varepsilon_0(\varepsilon)< \varepsilon$. From now on, we shall simply write $\varepsilon_0$ instead of $\varepsilon_0(\varepsilon)$. We define next
\begin{equation}\label{etaprima}
\eta'(\varepsilon):= \min \left\{ \frac{\varepsilon_0^2 (n_K(X))^2}{72},\, \frac{\left( \eta \left( \frac{\varepsilon_0}{3}\right) \right)^2 (n_K(X))^2}{72}\right \} \qquad \bigl(\varepsilon\in(0,1)\bigr),
\end{equation}
where $\eta$ is the function appearing in the hypotheses of the lemma. We fix $T\in \mathcal{K}(X)$ with $\nu(T)=1$ (thus, $\|T\|\leq \frac{1}{n_K(X)}$) and $(x_1,\, x_1^*)\in \Pi(X)$ such that
$$|x_1^* (T(x_1))|>1-\eta'(\varepsilon).$$
Since $T^*(B_{X^*})$ is relatively compact, we can find $x_2^*,\, \ldots,\, x_n^*\in B_{X^*}$ such that
$$\min_{2\leq j\leq n} \| T^* (x^*) - x_j^* \| < \eta'(\varepsilon) \quad \text{ for all } x^*\in B_{X^*}.$$
Similarly, since $T(B_X)$ is relatively compact, we can find $x_2,\, \ldots,\, x_\ell\in B_X$ such that
$$\min_{2\leq j\leq \ell} \| T(x) - x_j\| < \eta'(\varepsilon) \quad \text{ for all } x\in B_X.$$
Let $\widetilde{P}\colon X\longrightarrow \widetilde{P}(X)$, $i\colon \widetilde{P}(X)\longrightarrow X$ and $P:=i\circ \widetilde{P}\colon X\longrightarrow X$ satisfying the conditions (i)-(v) for $x_1,\ldots,x_\ell\in B_X$, $x_1^*,\ldots,x_n^*\in B_{X^*}$ and $\delta=\eta'(\varepsilon)$.

Now, for every $x^*\in B_{X^*}$, we have
\begin{multline*}
\| T^* (x^*) - P^* (T^*(x^*)) \| \\ \leq \min_{2\leq j\leq n} \left \{ \| T^* (x^*) - x_j^*\| + \| x_j^* - P^* (x_j^*)\| + \|P^*(x_j^*)-  P^* (T^* (x^*))\| \right\}
< 3\eta'(\varepsilon),
\end{multline*}
and hence,
$
\|T-TP\| = \|T^* - P^* T^*\| \leq 3\eta'(\varepsilon)
$. On the other hand, for each $x\in B_X$, we have
\begin{align*}
\| T(x) - P (T (x))\|  &\leq \min_{2\leq j\leq \ell} \left\{ \| T(x) - x_j\| + \|x_j - P (x_j)\| + \| P (x_j) - P (T (x))\| \right\} \\ &< 3\eta'(\varepsilon),
\end{align*}
and then,
$\| T - P T\| \leq 3 \eta'(\varepsilon)$. Therefore,
$$
\| P T P - T\| \leq \| P TP - P T\| + \| P T - T\| \leq \| T P - T\| + \| P T - T\| \leq 6\eta'(\varepsilon).
$$

Consider $(\widetilde{P} (x_1),\, i^* (x_1^*))\in \widetilde{P}(X) \times (\widetilde{P}(X))^*$. Note that  it is not true in general that $(\widetilde{P} (x_1),\, i^* (x_1^*))\in \Pi(\widetilde{P}(X))$, but we have that $\| \widetilde{P} (x_1) \| \leq 1$, $\| i^* (x_1^*) \| \leq 1$, and also, that
$$
x_1^*(i(\widetilde{P}(x_1))) = \underbrace{x_1^*(x_1)}_{=1} - \underbrace{x_1^*(i(\widetilde{P}(x_1)) - x_1)}_{\| P x_1 - x_1\| < \eta'(\varepsilon)} \quad \Longrightarrow \quad \operatorname{Re}(x_1^*(i(\widetilde{P}(x_1)))) \geq 1-\eta'(\varepsilon).
$$
By the Bishop-Phelps-Bollob\'{a}s Theorem (see \cite[Corollary 2.4.b]{CKMMPRB} for this version), there exist $(y,\, y^*)\in \Pi(\widetilde{P}(X))$ satisfying that
$$\max \left\{ \| y - \widetilde{P} (x_1)\|,\, \| y^* - i^* (x_1^*)\| \right\} \leq \sqrt{2\eta'(\varepsilon)} \leq \frac{\varepsilon_0}{3}.
$$
Next, we observe that the following two inequalities hold:
\begin{equation}\label{cota1}
\begin{split}
\| \widetilde{P}^*(y^*) - x_1^*\|  & \leq \|\widetilde{P}^*(y^*) - \widetilde{P}^*(i^*(x_1^*)) \| + \| \widetilde{P}^*(i^*(x_1^*)) - x_1^*\| \\ & \leq \sqrt{2 \eta'(\varepsilon)} + \eta'(\varepsilon)\leq \frac{2}{3} \varepsilon_0.
\end{split}
\end{equation}
\begin{equation}\label{cota2}
\| i(y) - x_1\| \leq \| i(y) - i(\widetilde{P}(x_1))\| + \| i(\widetilde{P}(x_1)) - x_1\| \leq \sqrt{2 \eta'(\varepsilon)} + \eta'(\varepsilon)\leq \frac{2}{3} \varepsilon_0.
\end{equation}

Let $T_1:=\widetilde{P}\circ T\circ i\colon \widetilde{P}(X) \longrightarrow \widetilde{P}(X)$.

\noindent\emph{Claim.} We have that
$$
|y^*(T_1 y)|> 1 - \eta\left( \frac{\varepsilon_0}{3} \right) \quad \text{and} \quad |y^*(T_1 y)|> 1 - \varepsilon_0.
$$
Indeed, from equations \eqref{cota1} and \eqref{cota2}, we obtain that
\begin{align*}
\bigl|x^*(T(x_1)) -& \widetilde{P}^*(y^*(T(i(y))))\bigr|  \\
&\leq |x_1^*(T(x_1)) - x_1^*(T(i(y)))| + |x_1^* (T(i(y))) - \widetilde{P}^*(y^*(T(i(y))))|  \\
&\leq \| T \| \| x_1 - i(y)\| + \| T\| \| x_1^* - \widetilde{P}^*(y^*)\| \\
&\leq 2\| T\| \left( \sqrt{2\eta'(\varepsilon)} + \eta'(\varepsilon)\right).
\end{align*}
Now, we can estimate $|y^*(T_1(y))|$ as follows:
\begin{align*}
|y^* (T_1(y)) | &= \bigl|\widetilde{P}^*(y^*(T(i(y))))\bigr|  \\  &\geq \bigl|x_1^*(T(x_1))\bigr| - \bigl|x_1^*(T(x_1)) - \widetilde{P}^*(y^*(T(i(y))))\bigr| \\ &\geq 1-\eta'(\varepsilon) - 2\| T \| \sqrt{2 \eta'(\varepsilon)} - 2\| T\| \eta'(\varepsilon).
\end{align*}
From here, using the definition of $\eta'(\varepsilon)$ given in Eq.~\eqref{etaprima} and the fact that $\|T\|\leq 1/n_K(X)$, we get both assertions of the claim.

In particular, we get that $\nu(T_1)\geq 1-\varepsilon_0>0$. On the other hand, we also have that $\nu(T_1)\leq 1$. Indeed, if there were some $(q,\, q^*)\in \Pi(\widetilde{P}(X))$ with $|q^*(T_1(q))|>1$, we would get
$$|q^*(T_1(q))| = |q^*(\widetilde{P}(T(i(q))))|=|(\widetilde{P}^*(q^*))(T(i(q)))| > 1,$$
but $\nu(T)=1$, and
$$(\widetilde{P}^*(q^*)) (i(q)) = q^*(\widetilde{P}(i(q))) = q^*(q)=1.$$
Thus $(i(q),\, \widetilde{P}^*(q^*))\in \Pi(X)$, and that is a contradiction.

We define now the operator $\widetilde{T}:= \frac{T_1}{\nu(T_1)}$. Clearly, $\widetilde{T}$ is a compact operator such that $\nu(\widetilde{T})=1$. From the claim, we get that
$$
\bigl|y^*(\widetilde{T}(y))\bigr| =\frac{1}{\nu(T_1)} |y^*(T_1(y))| \geq |y^*(T_1(y))| > 1- \eta \left( \frac{\varepsilon_0}{3} \right).
$$
Now, since $\widetilde{P}(X)$ has the BPBp-nu for compact operators with the mapping $\eta$, there exist a compact operator $\widetilde{S}\colon  \widetilde{P}(X) \longrightarrow \widetilde{P}(X)$  with $\nu(\widetilde{S}) = 1$ and $(z,\, z^*)\in \Pi(\widetilde{P}(X))$ such that
$$
\nu(\widetilde{S}) = \bigl|z^*(\widetilde{S}(z))\bigr|=1,\quad \| z-y\| < \frac{\varepsilon_0}{3},\quad \| z^* - y^*\| < \frac{\varepsilon_0}{3},\quad \| \widetilde{S} - \widetilde{T} \| < \frac{\varepsilon_0}{3}.
$$
Let $t=i(z)\in B_X$ and $t^*=\widetilde{P}^*(z^*)\in B_{X^*}$. We have that
$$t^*(t) = z^*(\widetilde{P}(i(z))) = z^*(z) = 1.$$
Thus $(t,\, t^*)\in \Pi(X)$, and also, by \eqref{cota1} and \eqref{cota2},
$$\| t - x_1\| \leq \| t - i(y)\| + \|i(y) - x_1\| = \| i(z) - i(y)\| + \| i(y) - x_1\| < \frac{\varepsilon_0}{3} + \frac{2\varepsilon_0}{3} = \varepsilon_0\leq \varepsilon,$$
$$\| t^* - x_1^*\| \leq \| \widetilde{P}^*(z^*) - \widetilde{P}^*(y^*)\| + \| \widetilde{P}^*(y^*) - x_1^*\| < \frac{\varepsilon_0}{3} + \frac{2\varepsilon_0}{3} = \varepsilon_0\leq \varepsilon.$$
We define $S=i\circ \widetilde{S} \circ \widetilde{P}\colon X\longrightarrow X$, which is a compact operator. It is clear that $\nu(S)\geq 1$ since
$$|t^* (S(t))| = |z^*(\widetilde{P}(i(\widetilde{S}(\widetilde{P}(i(z))))))| = |z^*(\widetilde{S}(z))|=1.$$
Also,
\begin{align*}
\| S - T\| &= \| i\circ \widetilde{S} \circ \widetilde{P} - T\| \\
&\leq \|i\circ \widetilde{S}\circ \widetilde{P} - i\circ \widetilde{T}\circ \widetilde{P}\| + \| i\circ \widetilde{T}\circ \widetilde{P} - P T P\| + \| P T P - T\| \\
&= \|i\circ \widetilde{S}\circ \widetilde{P} - i\circ \widetilde{T}\circ \widetilde{P}\| + \left\| \frac{P T P}{\nu(T_1)} - P T P\right\| + \| P T P - T\| \\
&\leq \| \widetilde{S} - \widetilde{T}\| + \|T\| \cdot \left| \frac{1}{\nu(T_1)}-1 \right| + \| P T P - T\|
\intertext{and, since $\| T \| \leq \frac{1}{n_K(X)}$, $1-\varepsilon_0 \leq \nu(T_1)\leq 1$, and $6\eta'(\varepsilon)\leq \frac{\varepsilon_0}{3}$, we continue as:}
& \leq \frac{\varepsilon_0}{3} + \frac{\varepsilon_0}{(1-\varepsilon_0)n_K(X)} + 6\eta'(\varepsilon) \leq \varepsilon_0\left( \frac{2}{3}+\frac{1}{(1-\varepsilon_0)n_K(X)} \right)< \varepsilon.
\end{align*}
We finish the proof if we prove that $\nu(S)\leq 1$. We consider the following cases:
\begin{itemize}
\item Case 1: if $\widetilde{P}$ is an absolute projection and $i$ is the natural inclusion, as a consequence of \cite[Lemma 3.3]{CMM}, we get that
$$\nu(S)=\nu(i\circ \widetilde{S}\circ \widetilde{P}) = \nu(\widetilde{S}) = 1.$$
\item Case 2: if $n_K(X)=n_K(\widetilde{P}(X))=1$, then
$$\nu(S) = \| S\| \leq \| \widetilde{S} \| = \nu(\widetilde{S}) = 1.$$
\end{itemize}
Hence, the result follows in the two cases.
\end{proof}

We will now provide some applications and consequences of the previous lemma. Given a continuous projection $P\colon X \longrightarrow X$, if we set $\widetilde{P}\colon X\longrightarrow \widetilde{P}(X)=P(X)\subset X$ (that is, $\widetilde{P}$ is just the operator $P$ with a restricted codomain) and $i\colon P(X) \longrightarrow X$ is the natural inclusion then, trivially, we have that $P=i\circ \widetilde{P}$ and that $\widetilde{P}\circ i=\mathrm{Id}_{\widetilde{P}(X)}$. This easy observation allows to get the following particular case of Lemma~\ref{LemaBPBpnuK1}.

\begin{prop}\label{Prop:Projec}
Let $X$ be a Banach space with $n_K(X)>0$. Suppose that there exists a net $\{ P_{\alpha} \}_{\alpha\in \Lambda}$ of norm-one projections on $X$ satisfying that $\{P_\alpha (x)\} \longrightarrow x$ for all $x\in X$ and $\{P_\alpha^*(x^*)\} \longrightarrow x^*$ for all $x^*\in X^*$, and that there exists a function $\eta\colon  (0,1) \longrightarrow (0,1)$ such that all the spaces $P_\alpha (X)$ with $\alpha\in \Lambda$ have the BPBp-nu for compact operators with the function $\eta$.
Suppose, moreover, that for each $\alpha\in \Lambda$, at least one of the following conditions is satisfied:
\begin{enumerate}
\item[(1)] the projection $P_\alpha$ is absolute,
\item[(2)] $n_K(P_\alpha(X))=n_K(X)=1$.
\end{enumerate}
Then, the space $X$ has the BPBp-nu for compact operators.
\end{prop}

We may now obtain the following consequence of the above result. Given a Banach space $X$ and $m\in \mathbb{N}$, the space $\ell_\infty^m(X)$ represents the $\ell_\infty$-sum of $m$ copies of $X$, and we will write $\ell_\infty(X)$ for the $\ell_\infty$-sum of countably infinitely many copies of $X$. Similarly, $c_0(X)$ is the $c_0$-sum of countably infinitely many copies of $X$. When $X=\mathbb{K}$, we just write $\ell_\infty^m$ for $\ell_\infty^m(\mathbb{K})$.

\begin{cor}\label{CorEquiv}
Let $X$ be a Banach space with $n_K(X)>0$. Then, the following statements are equivalent:
\begin{enumerate}
\item[(i)] The space $c_0(X)$ has the BPBp-nu for compact operators.
\item[(ii)] There is a function $\eta\colon  (0,1) \longrightarrow (0,1)$ such that all the spaces $\ell_{\infty}^n(X)$, with $n\in\mathbb{N}$, have the BPBp-nu for compact operators with the function $\eta$.
\end{enumerate}
Moreover, if $X$ is finite dimensional, these properties hold whenever $c_0(X)$ or $\ell_{\infty}(X)$ have the BPBp-nu.
\end{cor}

\begin{proof}
That (ii) implies (i) is a consequence of Proposition \ref{Prop:Projec} since for every $n\in \mathbb{N}$, the operator on $c_0(X)$ which is the identity on the first $n$ coordinates and $0$ elsewhere is an absolute projection whose image is isometrically isomorphic to  $\ell_{\infty}^n(X)$.

(i) implies (ii) is a consequence of \cite[Proposition 4.3]{CDJM}, as one can easily see $\ell_\infty^n(X)$ as an $\ell_\infty$-summand of $c_0(X)$. Let us comment that the function $\eta$ valid for all $\ell_\infty^n(X)$ is the function valid for $c_0(X)$ and this actually follows from the proof of \cite[Theorem 4.1]{CDJM} (from which \cite[Proposition 4.3]{CDJM} actually follows).

Finally, when $X$ has finite dimension, if $c_0(X)$ or $\ell_{\infty}(X)$ has the BPBp-nu, then condition (ii) holds by using \cite[Theorem 4.1]{CDJM} and the fact that $\ell_\infty^n(X)$ is finite-dimensional and so, every operator from $\ell_{\infty}^n(X)$ to itself is compact.
\end{proof}

As stated in Examples \ref{EjemplosIniciales}, that $c_0$ and the spaces $\ell_\infty^n$ for $n\in \mathbb{N}$ have the BPBp-nu for compact operators is a consequence of \cite[Corollary 4.2]{GK} and \cite[Proposition 2]{KLM14}. Actually, the fact that all the space $\ell_\infty^n$ have the BPBp-nu with the same function $\eta$ follows from \cite[Corollary 4.2]{GK} and (the proof of) \cite[Theorem 4.1]{CDJM}. However, let us note that we can also get this result as a consequence of our previous corollary.

\begin{cor}\label{Corc0}
There is a function $\eta\colon  (0,1) \longrightarrow (0,1)$ such that the space $c_0$ and the spaces $\ell_{\infty}^n$ with $n\in \mathbb{N}$, have the BPBp-nu for compact operators with the function $\eta$.
\end{cor}

Additionally, \cite[Proposition 4.3]{CDJM} also implies that whenever $\ell_{\infty}^n(X)$ has the BPBp-nu for compact operators for some $n\in \mathbb{N}$, then so does $X$, although the converse remains unknown in general (even for $n=2$).

Another consequence of Proposition \ref{Prop:Projec} is the following:

\begin{cor}\label{Cor:ProjSubsets}
Let $X$ be a Banach space with $n_K(X)>0$. Suppose that there exists a net $\{ P_{\alpha} \}_{\alpha\in \Lambda}$ of norm-one projections on $X$ such that $\alpha \preceq \beta$ implies $P_\alpha(X) \subset P_\beta(X)$, that $\{ P_\alpha^*(x^*)\} \longrightarrow x^*$ for all $x^*\in X^*$, and that there exists a function $\eta\colon  (0,1) \longrightarrow (0,1)$ such that all the spaces $P_\alpha (X)$ with $\alpha\in \Lambda$ have the BPBp-nu for compact operators with the function $\eta$.
Suppose, moreover, that for each $\alpha\in \Lambda$, at least one of the following conditions is satisfied:
\begin{enumerate}
\item[(1)] the projection $P_\alpha$ is absolute,
\item[(2)] $n_K(P_\alpha(X))=n_K(X)=1$.
\end{enumerate}
Then, the space $X$ has the BPBp-nu for compact operators.
\end{cor}

\begin{proof}
Observe that in order to apply Proposition \ref{Prop:Projec} we only need that $\{ P_\alpha x\} \longrightarrow x$ in norm for all $x\in X$. But this is proved in  \cite[Corollary 2.4]{DGMM}, so we are done.
\end{proof}

The previous result can be used to prove that all the preduals of $\ell_1$ have the BPBp-nu for compact operators.

\begin{cor}\label{Cor:PredualEll1}
Let $X$ be a Banach space such that $X^*$ is isometrically isomorphic to $\ell_1$. Then $X$ has the BPBp-nu for compact operators.
\end{cor}

\begin{proof}
By using a deep result due to Gasparis \cite{Gasparis}, it is shown in the proof of \cite[Theorem 3.6]{DGMM} that there exists a sequence of norm-one projections $P_n\colon X\longrightarrow X$ satisfiying that $P_{n+1}P_n=P_n$ (and so, $P_n(X) \subset P_{n+1}(X)$), that $P_n(X)$ is isometrically isomorphic to $\ell_\infty^n$, and also that $P_n^*(x^*)\longrightarrow x^*$ for all $x^*\in X^*$ (this claim holds since the sets $Y_n$ defined on that proof satisfy that their union is dense in $X^*=\ell_1$).

Next, as $P_n(X)$ is isometrically isomorphic to $\ell_\infty^n$, on the one hand we have that all the spaces $P_n(X)$ have the BPBp-nu for compact operators with the same function $\eta$ as a consequence of Corollary \ref{Corc0}. On the other hand, $n(X)=n(P_n(X))=1$ for all $n\in \mathbb{N}$ (see \cite{KaMaPa}, for instance) so, in particular, $n_K(X)=n_K(P_n(X))=1$ for all $n\in \mathbb{N}$. Finally, Corollary \ref{Cor:ProjSubsets} provide the desired result.
\end{proof}

\section[Second ingredient]{Second ingredient: a strong approximation property of
\texorpdfstring{$C_0(L)$}{C0(L)} spaces and their duals}\label{SectionCL}

The aim of this section is to provide some strong approximation property of $C_0(L)$ spaces and their duals which allow to use Lemma \ref{LemaBPBpnuK1} (actually, Proposition \ref{Prop:Projec}) to give a proof of Theorem \ref{TheoremMain}. We need a number of technical lemmas.

\begin{lemma}\label{Lemma:Topol}
Let $L$ be a locally compact space, let $\{ K_1, \ldots, K_M \}$ be a family of pairwise disjoint non-empty compact subsets of $L$, and let $K\subset L$ be a compact set with $\bigcup\limits_{m=1}^M K_m \subset K$. If $\{ U_1,\ldots, U_R \}$ is a family of relatively compact open subsets of $L$ covering $K$ such that for each $m$ there is an $r(m)$ with $K_m\subset U_{r(m)}$, $m=1,\ldots, M$, then there exists an open refinement $\{ Z_1,\ldots, Z_S \}$, $M\leq S\leq R+M$ with $Z_1,\ldots, Z_M$ pairwise disjoint, satisfying:
\begin{enumerate}[(1)]
\item For $m=1,\ldots, M$, $K_m\subset Z_m$, and $K_m\cap Z_s=\emptyset$ for all $s\in \{ 1,\ldots, S \} \backslash \{ m\}$.
\item For all $s_0>M$, there exists $z_{s_0}\in Z_{s_0}\backslash \left( \bigcup\limits_{s\neq s_0} Z_s \right)$.
\end{enumerate}
\end{lemma}

\begin{proof}
As $\{ K_1, \ldots, K_M\}$ are pairwise disjoint, there exist $\{ V_1,\ldots V_M\}$ pairwise disjoint open subsets of $L$ with $K_m\subset V_m\subset U_{r(m)}$, $m=1,\ldots, M$.

The family $\left\{ V_1, \ldots, V_M, U_1\backslash \left( \bigcup\limits_{m=1}^M K_m\right), \ldots, U_R \backslash \left( \bigcup\limits_{m=1}^M K_m\right)\right\}$ is another cover of $K$ by open subsets of $L$ subordinated to $\{ U_r\}_{r=1}^R$. We define $Z_m:= V_m$ for $m=1,\ldots,M$, and  $W_r:= U_r \backslash \left( \bigcup\limits_{m=1}^M K_m\right)$ for $r=1,\ldots, R$.

If $W_1 \subset V_1\cup \ldots \cup V_M$, then $\{ V_1, \ldots, V_M, W_2, \ldots, W_R\}$ is again a cover of $K$. If that happens again and again until $W_R$, we have that $\{ Z_1, \ldots, Z_M\}$ is the cover we were looking for. In other case, let $r_1\geq 1$ be the first natural number such that there exists  $w_{r_1}\in W_{r_1}\backslash \left( \bigcup\limits_{m=1}^M V_m\right)$, and denote $Z_{M+1}:=W_{r_1}$. The family $\{ V_1, \ldots, V_M, W_{r_1},W_{r_1+1}, \ldots, W_{R}\}$ is a cover of $K$ by open sets, and then, so is the family $$\bigl\{ V_1, \ldots, V_M, W_{r_1}, W_{r_1+1}\backslash \{ w_{r_1}\}, \ldots, W_{R}\backslash \{ w_{r_1}\}\bigr\}.$$
Consider now $r_2>r_1$ the first natural number such that there exists  $w_{r_2}\in W_{r_2}\backslash \{ w_{r_1}\}$ and $w_{r_2}\notin V_1\cup \ldots \cup V_M \cup W_{r_1}$. Let $Z_{M+2}:= W_{r_2}\backslash \{ w_{r_1}\}$ and proceed as before. In at most $R$ steps, we get $\{ Z_1, \ldots, Z_S\}$, $M\leq S\leq R+M$, such that
\begin{itemize}
\item $K_m\subset Z_m$ for $m=1,\ldots M$.
\item $\left( \bigcup\limits_{m=1}^M K_m\right)\cap Z_s=\emptyset$ for $s>M$.
\item For all $s_0>M$, there exists $w_{r_{s_0-M}}\in Z_{s_0}\backslash \left( \bigcup\limits_{s\neq s_0} Z_s \right)$.\qedhere
\end{itemize}
\end{proof}

We next provide a result showing the existence of certain partitions of the unity. We separate the non-compact case (Lemma \ref{Lemma:PartUni}) and the compact case (Lemma \ref{Lemma:PartUniComp}) for the sake of clarity. We start with the non-compact case.

\begin{lemma}\label{Lemma:PartUni}
Let $L$ be a non-compact locally compact space. Let $K\subset L$ be a compact set and $\{ K_1,\ldots, K_M\}$ a  family of pairwise disjoint non-empty compact subsets of $K$. Given  a  family $\{ U_1, \ldots, U_R\}$ of relatively compact open subsets of $L$ that cover $K$, let $\{ Z_1, \ldots, Z_S\}$ be a family of open subsets of $L$ covering $K$ such that they satisfy the thesis of Lemma~\ref{Lemma:Topol}, and denote by $Z_{S+1}$ the set $L\backslash \left( \bigcup\limits_{s=1}^S Z_s\right)$. Then, there exists a partition of the unity subordinated to $\{ Z_s\}_{s=1}^{S+1}$, $\{ \varphi_s\}_{s=1}^{S+1}$, such that:
\begin{enumerate}
\item[(1)] $\{ \varphi_1,\ldots, \varphi_{M}\}$ have disjoint support.
\item[(2)] $\varphi_m(K_m)\equiv 1$, for $m=1,\ldots, M$.
\item[(3)] For all $M<s\leq S+1$, there exists $z_s\in Z_s$ such that $\varphi_s(z_s)=1$.
\item[(4)] For $s=1,\ldots,S+1$, $\supp(\varphi_s)\subset Z_s$.
\item[(5)] $(\varphi_1 + \cdots + \varphi_S)(x)=1$, for all $x\in K$.
\end{enumerate}
\end{lemma}

\begin{proof}
By hypothesis, there exists some $z_{S+1}\in L\backslash \left( \bigcup\limits_{s=1}^S Z_s\right)$, since for all $s$, $\overline{Z_s}\subset \overline{\bigcup\limits_{r=1}^R U_r}$, which is a compact set. Now, we follow the argument from the proof of \cite[Theorem 2.13]{R}, but adapted to our case.

As $K\subset Z_1 \cup \ldots \cup Z_S$, for each $x\in K$, there exists a neighbourhood of $x$, $Y_x$, with compact closure $\overline{Y_x}\subset Z_s$ for some $s$. Consider $x_1, \ldots, x_p$ such that $K\subset Y_{x_1}\cup \ldots \cup Y_{x_p}$. For each $1\leq s\leq S$, let $H_s$ be the union of those $\overline{Y_{x_j}}$ which lie in $Z_s$, and if $M<s_0\leq S$, we take $H_{s_0}\cup \{z_{s_0}\}$, with $z_{s_0}\in Z_{s_0}\backslash \left( \bigcup\limits_{s\neq s_0} Z_s\right)$. Note that the sets $H_1,\ldots, H_M$ and $H_{M+1}\cup\{z_{M+1}\}, \ldots, H_{S}\cup\{z_{S}\}$ are non-empty. By Urysohn's Lemma, there are continuous functions $g_s\colon L\longrightarrow [0, 1]$ such that $g_s(H_s)\equiv 1$ and $g_s\big|_{L\backslash Z_s} \equiv 0$, for $1\leq s\leq M$, and $g_{s_0}(H_{s_0}\cup \{z_0\})\equiv 1$ and $g_{s_0}\big|_{L\backslash Z_{s_0}}\equiv 0$ for $M<s_0\leq S$. Define
\begin{align*}
\varphi_1 & := g_1,\\
\varphi_2 & := (1-g_1)g_2, \\
& \vdots \\
\varphi_S & := (1-g_1)(1-g_2)\cdots (1-g_{S-1})g_S
\end{align*}
Clearly, $\supp(\varphi_s)\subset Z_s$ for all $s=1,\ldots, S$, and we have that $$\varphi_1+\cdots+\varphi_S=1-(1-g_1)\cdots(1-g_S).$$
Since $K\subset H_1\cup \ldots \cup H_S$, for each $x\in K$, there exists $s=s(x)$ with $g_s(x)=1$, and also, for all $s=1, \ldots, M$, we have that $$\{ x\in L\colon  \varphi_s(x)\neq 0\} \subset \{ x\in L\colon  g_s(x)\neq 0\} \subset Z_s.$$
Therefore, the functions $\{ \varphi_1,\ldots, \varphi_M\}$ have disjoint support, and $\varphi_1+\cdots+\varphi_S\equiv 1$ on $K$.

We define $\varphi_{S+1}:=1-(\varphi_1+\cdots+\varphi_S)=(1-g_1)\cdots (1-g_S)$. Moreover, $K_m\subset Z_m$ for $m=1, \ldots, M$, and $K_m\cap Z_s=\emptyset$ for $m\neq s$, $m=1,\ldots, M$, $s=1,\ldots, S$. Hence,
$$\varphi_m(x)=\sum_{s=1}^S \varphi_s(x)=1, \quad \forall x\in K_m,\ m=1,\ldots, M.$$
On the other hand, if $M<s_0\leq S$, let $z_{s_0}\in Z_{s_0}\backslash \left( \bigcup\limits_{s\neq s_0} Z_s\right)$. We have that
$$\varphi_{s_0}(z_{s_0})=\sum_{s=1}^S \varphi_s(z_{s_0})=1,$$
and
\begin{equation*}
z_{S+1}\notin \bigcup\limits_{s=1}^{S} Z_s,\, \text{ thus } \varphi_{S+1}(z_{S+1})=1.\qedhere
\end{equation*}
\end{proof}

The next result is the version of the previous lemma for compact topological spaces.

\begin{lemma}\label{Lemma:PartUniComp}
Let $L$ be a compact space. Let $\{ K_1,\ldots, K_M\}$ be a  family of pairwise disjoint non-empty compact subsets of $L$. Given a family $\{ U_1, \ldots, U_R\}$ of relatively compact open subsets of $L$ that cover it, let $\{ Z_1, \ldots, Z_S\}$ be a family of open subsets of $L$ covering $K$ such that they satisfy the thesis of Lemma \ref{Lemma:Topol}. Then, there exists a partition of the unity subordinated to $\{ Z_s\}_{s=1}^{S}$, $\{ \varphi_s\}_{s=1}^{S}$, such that:
\begin{enumerate}
\item[(1)] $\{ \varphi_1,\ldots, \varphi_{M}\}$ have disjoint support.
\item[(2)] $\varphi_m(K_m)\equiv 1$, for $m=1,\ldots, M$.
\item[(3)] For all $M<s\leq S$, there exists $z_s\in Z_s$ such that $\varphi_s(z_s)=1$.
\item[(4)] For $s=1,\ldots,S$, $\supp(\varphi_s)\subset Z_s$.
\item[(5)] $(\varphi_1 + \cdots + \varphi_S)(x)=1$, for all $x\in K$.
\end{enumerate}
\end{lemma}

\begin{proof}
We can follow the proof of Lemma \ref{Lemma:PartUni} taking $K=L$ and adapting the steps from that proof, keeping in mind that now $Z_{S+1}=\emptyset$ (and hence there is not such a point $z_{S+1}$), and that the mapping $\varphi_{S+1}$ is  identically $0$, and hence, it can be omitted.
\end{proof}

The following result provides the promised approximation property of $C_0(L)$ spaces and their duals.

\begin{theorem}\label{theorem:AP_C0(L)}
Let $L$ be a locally compact space. Given $\{ f_1,\ldots , f_\ell \}\subset C_0(L)$ such that $\|f_j\|\leq 1$ for $j=1,\ldots, \ell$, and given $\{ \mu_1,\ldots, \mu_n\}\subset C_0(L)^*$ with $\| \mu_j\|\leq 1$ for $j=1,\ldots, n$, for each $\varepsilon >0$ there exists a norm one projection $P\colon C_0(L)\longrightarrow C_0(L)$ satisfying:
\begin{enumerate}
\item[(1)] $\| P^* (\mu_j) - \mu_j\| < \varepsilon$, for $j=1,\ldots, n$,
\item[(2)] $\| P (f_j) - f_j\| < \varepsilon$, for $j=1,\ldots, \ell$,
\item[(3)] $P(C_0(L))$ is isometrically isomorphic to $\ell_{\infty}^p$ for some $p\in \mathbb{N}$.
\end{enumerate}
\end{theorem}

Let us comment that this result extends \cite[Lemma 3.4]{DGMM} (which, actually, was itself an extension of \cite[Proposition 3.2]{ABGCCKLLM} and \cite[Proposition 3.2]{JW}). The main difference is that here we are able to deal with an arbitrary number of functions of $C_0(L)$ in (2), while in that lemma only one function is controlled, and besides, this was done with the help of an inclusion operator which is not the canonical one. However, this difference is crucial in order to apply Lemma~\ref{LemaBPBpnuK1} (or even its consequence Proposition~\ref{Prop:Projec}).

The following observation on the theorem is worth mentioning.

\begin{remark}\label{remark:pi-property}
Let us observe that by just conveniently ordering the obtained projections in Theorem~\ref{theorem:AP_C0(L)}, we actually get the following: {\slshape given a Hausdorff locally compact topological space $L$, there is a net $\{ P_{\alpha} \}_{\alpha\in \Lambda}$ of norm-one projections on $C_0(L)$, converging in the strong operator topology to the identity operator, such that $\{ P_{\alpha}^*\}_{\alpha\in \Lambda}$ converges in the strong operator topology to the identity on $C_0(L)^*$, and such that $P_\alpha(C_0(L))$ is isometrically isomorphic to a finite-dimensional $\ell_\infty$ space.}
\end{remark}

\begin{proof}[Proof of Theorem~\ref{theorem:AP_C0(L)}]
We will assume first that $L$ is not compact. Since $f_j\in C_0(L)$, $j=1,\ldots, \ell$, there exists a compact set $K_0\subset L$ such that
$$\sup_{j=1,\ldots,\ell}\{ |f_j(x)|\colon  x\in L\backslash K_0\} < \frac{\varepsilon}{4}.$$
For each $x\in K_0$, there exists a relatively compact open subset $U_x$ of $L$ containing $x$ and such that $$|f_j(x) - f_j(y)| < \frac{\varepsilon}{2}\quad \text{ for $y\in U_x$ and $j=1,\ldots, \ell$.}
$$
Therefore, $\{ U_x\}_{x\in K_0}$ is a cover of $K_0$, and so, there exist a finite subcover $\{ U_1,\ldots, U_{R-1}\}$ such that $K_0\subseteq U_1\cup \ldots \cup U_{R-1}$, and if $x, y\in U_r$ for some $r$, then $|f_j(x)-f_j(y)|<\frac{\varepsilon}{2}$, for $j=1,\ldots, \ell$.

We define $\mu:=\sum_{j=1}^n |\mu_j|\in C_0(L)^*$. Since for each $j\in \{1,\ldots,n\}$ $\mu_j$ is absolutely continuous with respect to $\mu$, by the Radon-Nikod\'ym Theorem, there exists $g_j\in L_1(\mu)$ such that $\mu_j=g_j \mu$, that is,
$$
\mu_j(f):= \int_L f \text{d} \mu_j = \int_L f(x) g_j(x) \text{d}\mu(x)\quad \text{ for all $f\in C_0(L)$.}
$$
Since the set of simple functions is dense in $L_1(\mu)$, we may choose a set of simple functions $\{ s_j\colon  j=1,\ldots, n\}$ such that $\| g_j-s_j\|_1<\frac{\varepsilon}{4}$ for $j=1,\ldots,n$.

Next, we consider a family $\{A_m\}_{m=1}^M$ of pairwise disjoint measurable sets with $\mu(A_m)>0$ for all $m$, such that each $A_m$ is contained in one of the elements of the following cover of $L$: $\{U_1, \ldots, U_{R-1}, L\backslash K_0\}$, and also  $\{\alpha_{m,j}\colon m=1,\ldots,M,\, j=1,\ldots, n\}$ such that $s_j=\sum_{m=1}^{M}\alpha_{m,j}\chi_{A_m}$. This cover satisfies that if $x, y\in L\backslash K_0$, or if $x,y\in U_r$, then $|f_j(x)-f_j(y)|<\frac{\varepsilon}{2}$ for all $j=1,\ldots,\ell$ and all $r=1,\ldots, R-1$. Let $C>\max\{|\alpha_{m,j}|\colon  m=1,\ldots,M,\, j= 1,\ldots,n\}$.

Since $\mu$ is regular, for each $1\leq m\leq M$, we can find a compact set $K_m\subset A_m$ such that $\mu(A_m\backslash K_m)<\frac{\varepsilon}{4MC}$ and $\mu(K_m)>0$ for all $m=1,\ldots, M$.

Let $K=K_0\cup K_1\cup\ldots \cup K_M$. As $K\backslash \left(\bigcup\limits_{r=1}^{R-1} U_r\right)$ is a compact subset of $L$, we can cover it with finitely many relatively compact open subsets of $L\backslash K_0$ that we will denote $U_R,U_{R+1},\ldots,U_P$. If we now apply Lemmas \ref{Lemma:Topol} and \ref{Lemma:PartUni} to the family $\{ U_1,\ldots, U_P\}$ and the compacts $\{ K_1,\ldots, K_M\}$ and $K$, we obtain a refinement of relatively compact open subsets of $L$, $\{ Z_1,\ldots, Z_S\}$ with $K_m\subset Z_m$ for $m=1,\ldots, M$ and $\{ Z_1,\ldots Z_M\}$ pairwise disjoint, and defining $Z_{S+1}$ to be the set $L\backslash \left( \bigcup\limits_{s=1}^S Z_s\right)$, we also have a partition of the unity subordinated to $\{ Z_s\}_{s=1}^{S}$, $\{ \varphi_s\}_{s=1}^{S+1}$, such that:
\begin{enumerate}
\item[(i)] $\{ \varphi_1,\ldots,\varphi_M\}$ have disjoint support.
\item[(ii)] $\varphi_m(K_m)\equiv 1$ for $m=1,\ldots, M$.
\item[(iii)] For all $M<s\leq S+1$, there exists $z_s\in Z_s$ such that $\varphi_s(z_s)=1$.
\item[(iv)] For $s=1,\ldots,S+1$, $\supp(\varphi_s)\subset Z_s$.
\item[(v)] $(\varphi_1+\ldots+\varphi_S)(K)\equiv 1$.
\end{enumerate}
Now, we define $P\colon C_0(L)\longrightarrow C_0(L)$ by
$$P(f):= \sum_{m=1}^M \frac{1}{\mu(K_m)} \left( \int_{K_m} f \, \text{d} \mu \right) \varphi_m + \sum_{s=M+1}^{S+1} f(z_s)\varphi_s,\quad \text{ for all $f \in C_0(L)$.}
$$
Let us first check that (2) holds, that is, that $\| P (f_j) - f_j\|<\varepsilon$ for all $j=1,\ldots, \ell$. Let $x\in L$. We will distinguish two cases:
\begin{itemize}
\item Case 1: if $x\in \bigcup\limits_{m=1}^M Z_m$, then there exists exactly one $m_0$ such that $x\in Z_{m_0}$. Then, for each $j=1,\ldots, \ell$, we have:
\begin{align*}
|P(f_j)(x)-f_j(x)| &= \left|P(f_j)(x) - \sum_{m=1}^M f_j(x)\varphi_m(x) - \sum_{s=M+1}^{S+1} f_j(x) \varphi_s(x)\right|  \\
&\leq \underbrace{\left|\frac{1}{\mu(K_{m_0})}\left( \int_{K_{m_0}} f_j(y) \, \text{d}\mu(y)\right) - f_j(x)\right|\varphi_{m_0}(x)}_{\text{(I)}} +\\
&\quad + \underbrace{\sum_{s=M+1}^{S+1} |f_j(x) - f_j(z_s)| \varphi_s(x)}_{\text{(II)}}.
\end{align*}
For $\text{(I)}$, we have
\begin{equation*}
\text{(I)} =\left|\frac{1}{\mu(K_{m_0})}\left( \int_{K_{m_0}} (f_j(y)-f_j(x)) \, \text{d}\mu(y)\right)\right|\varphi_{m_0}(x) \leq \frac{1}{\mu(K_{m_0})} \int_{K_{m_0}} \frac{\varepsilon}{2} \, \text{d} \mu(y) = \frac{\varepsilon}{2}.
\end{equation*}
Now, for $\text{(II)}$, let $s\in \{ M+1,\ldots, S+1\}$. Note that if $x\notin Z_s$, then $\varphi_s(x)=0$, and if $x\in Z_s$, we have that $|f_j(x) - f_j(z_s)|< \frac{\varepsilon}{2}$ and $\sum_{s=M+1}^{S+1}\varphi_s(x)\leq 1$, and so,
$\text{(II)} < \frac{\varepsilon}{2}$. Therefore, $|P(f_j)(x) - f_j(x)|< \varepsilon$ for all $x\in \bigcup\limits_{m=1}^M Z_m$, for all $j=1,\ldots, \ell$.
\item Case 2: if $x\notin \bigcup\limits_{m=1}^M Z_m$, then for each $j=1,\ldots,\ell$, we have
$$|P(f_j)(x) - f_j(x)| = \left| \sum_{s=M+1}^{S+1} (f_j(x) - f_j(z_s)) \varphi_s(x)\right| < \frac{\varepsilon}{2}$$ as in item $\text{(II)}$ of the previous case.
\end{itemize}
Summarizing, we get $\| P(f_j) - f_j\| < \varepsilon$ for all $j=1,\ldots, \ell$, getting thus (2).

Now we check (1), that is, that $\| P^*(\mu_j) - \mu_j\| < \varepsilon$ for all $j=1,\ldots,n$. Indeed, first observe that if $\nu$ is a regular Borel (real or complex) measure on $L$, its associated $x_\nu^*\in C_0(L)^*$ is defined as
$$x_\nu^*(f):= \int_L f(x) \, \text{d} \nu(x),\quad \forall f\in C_0(L),$$ and we identify $x_\nu^*\equiv \nu$. In our case, we have that
\begin{align*}
P^*(\nu)(f) &= \int_L P(f)(x) \, \text{d} \nu(x) \\
&= \int_L \left( \sum_{m=1}^M \frac{1}{\mu(K_m)} \left( \int_{K_m} f \, \text{d} \mu \right) \varphi_m(x) \right) \, \text{d} \nu(x) + \int_L \left( \sum_{s=M+1}^{S+1} f(z_s) \varphi_s(x) \right) \, \text{d} \nu(x) \\
&= \sum_{m=1}^M \frac{1}{\mu(K_m)} \left( \int_{K_m} f \, \text{d} \mu \right) \int_L \varphi_m(x)\, \text{d} \nu(x) + \sum_{s=M+1}^{S+1} f(z_s) \int_L \varphi_s(x) \, \text{d} \nu(x).
\end{align*}
In particular, if $\supp(\nu)\subset \bigcup\limits_{m=1}^{M} K_m$, then by Lemma \ref{Lemma:Topol}.(1)
$$\sum_{s=M+1}^{S+1} f(z_s) \int_L \varphi_s(x) \, \text{d} \nu(x) \equiv 0,\quad \forall f\in C_0(L).$$

Let now $\nu_j:= t_j \mu$, where $t_j:= \sum_{m=1}^M \alpha_{m,j} \chi_{K_m}$, for all $j=1,\ldots, n$, that is,
$$\nu_j(f)= \int_L f(x) \left( \sum_{m=1}^M \alpha_{m,j} \chi_{K_m}(x) \right) \, \text{d} \mu(x), \quad \forall f\in C_0(L).$$
It holds that $P^*(\nu_j)=\nu_j$ for $j=1,\ldots, n$. Indeed, as $\supp(\nu_j)\subset \bigcup\limits_{m=1}^{M}K_m$, we have
\begin{align*}
P^*(\nu_j)(f) &= \sum_{m=1}^M \frac{1}{\mu(K_m)} \left( \int_{K_m} f\, \text{d} \mu \right) \int_L \varphi_m(x) \left( \sum_{l=1}^M \alpha_{l,j} \chi_{K_l}(x) \right) \, \text{d} \mu(x) \\
&= \sum_{m=1}^M \frac{1}{\mu(K_m)} \left( \int_{K_m} f\, \text{d} \mu \right) \underbrace{\int_L \alpha_{m,j} \chi_{K_m}(x) \, \text{d} \mu(x)}_{\alpha_{m,j} \mu(K_m)} \\
&= \int_L f(x) \left( \sum_{m=1}^M \alpha_{m,j} \chi_{K_m}(x) \right) \, \text{d} \mu(x) = \nu_j(f)
\end{align*}
for all $f\in C_0(L)$ and all $j=1,\ldots, n$.

Now, we know that $\| P^*\| = \| P \| \leq 1$ and, since $P(\varphi_j)=\varphi_j$ for $j=1,\ldots,n$, we get that $\| P^*\|=1$. Therefore, since $P^*(\nu_j)=\nu_j$, we get
\begin{align*}
\| P^*(\mu_j) - \mu_j \| &\leq \| P^*(\mu_j - \nu_j) \| + \| \nu_j - \mu_j\| \\ &\leq \| P^* \| \cdot \| \mu_j - \nu_j \| + \| \mu_j - \nu_j \| \leq 2 \| \mu_j - \nu_j \|.
\end{align*}
But we have
\begin{align*}
\| \mu_j - \nu_j \| & = \| g_j \mu - t_j \mu \| \leq \| g_j \mu - s_j\mu \| + \| s_j\mu - t_j \mu \| \\ &= \| g_j - s_j \|_1 + \| s_j - t_j \|_1 < \frac{\varepsilon}{4} + \frac{\varepsilon}{4} = \frac{\varepsilon}{2},
\end{align*}
since
\begin{align*}
\| s_j - t_j\|_1 &= \int_L \left| \sum_{m=1}^M \alpha_{m,j} \chi_{A_m} - \sum_{m=1}^M \alpha_{m,j} \chi_{K_m} \right| \, \text{d}\mu  \\
&\leq \sum_{m=1}^M \underbrace{|\alpha_{m,j}|}_{\leq C} \mu(A_m \backslash K_m) < \frac{MC\varepsilon}{4 M C} = \frac{\varepsilon}{4},
\end{align*}
for all $j=1,\ldots, n$. Hence,
$$
\| P^*(\mu_j) - \mu_j\| \leq 2 \| \mu_j - \nu_j\| < 2\frac{\varepsilon}{2}=\varepsilon \quad \text{ for $j=1,\ldots, n$.}
$$
Let us finish the proof by checking (3). As $\mu(K_m)>0$, we have $K_m\neq \emptyset$, $m=1,\ldots, M$. Hence, we have that $z_s\in Z_s$ for $s=1,\ldots, S+1$ and that $z_{s_0}\notin \bigcup\limits_{s\neq s_0} Z_s$ for all $s_0=1,\ldots, S+1$. By the definition of $P$, we have that $P(C_0(L))=\operatorname{span} \{ \varphi_s\colon  s=1,\ldots, S+1\}$ and we will be done by proving the following equality:
$$
\bigl\| a_1 \varphi_1 + \cdots + a_{S+1} \varphi_{S+1}\bigr\|_{\infty} = \max \{ |a_1|,\ldots, |a_{S+1}|\} = \| a\|_{\infty}
$$
for every $a=(a_1,\ldots, a_{S+1})$. Indeed, for $x\in L$
$$
\bigl|a_1 \varphi_1(x) + \cdots + a_{S+1}\varphi_{S+1}(x)\bigr| \leq \| a \|_{\infty} \sum_{s=1}^{S+1} \varphi_s(x) = \| a \|_{\infty}.
$$
But for each $s$,
$$
\bigl|a_1 \varphi_1(z_s) + \cdots + a_{S+1}\varphi_{S+1}(z_s)\bigr| = |a_s|,$$
and then,
$$
\bigl\| a_1 \varphi_1 + \ldots + a_{S+1} \varphi_{S+1}\bigr\|_{\infty} \geq \| a\|_\infty.
$$
Hence, the mapping $\rho\colon  \ell_{\infty}^{S+1} \longrightarrow C_0(L)$ given by
$$(a_1,\ldots, a_{S+1}) \longmapsto a_1 \varphi_1 + \ldots + a_{S+1}\varphi_{S+1}$$
is an isometry, and therefore, $P(C_0(L))$  is isometrically isomorphic to  $\ell_{\infty}^{S+1}$.

Now, for the case when $L$ is compact, by taking $K_0=L$ and using Lemma \ref{Lemma:PartUniComp} instead of Lemma \ref{Lemma:PartUni}, a similar proof is valid, except that now all the elements depending on $S+1$ will vanish in the proof: here we get $Z_{S+1}=\emptyset$ (hence $z_{S+1}$ does not exist), $\varphi_{S+1}\equiv 0$ (and hence it can be omitted), and so, the vector $a$ will only have $S$ components; therefore $P(C_0(L))$  is isometrically isomorphic to $\ell_{\infty}^S$ in this case.
\end{proof}

We are now ready to prove the main result of the paper.

\begin{proof}[Proof of Theorem \ref{TheoremMain}]
Let $f_1,\ldots, f_{\ell}\in B_{C_0(L)}$, $\mu_1,\ldots, \mu_n\in B_{(C_0(L))^*}$ and $\varepsilon>0$ be given. Let $P\colon C_0(L)\longrightarrow C_0(L)$ be the projection from Theorem \ref{theorem:AP_C0(L)}, which satisfies that $P(C_0(L))$ is isometrically isomorphic to $\ell_\infty^p$ for some $p\in \mathbb{N}$. Let $\widetilde{P}\colon X\longrightarrow \widetilde{P}(C_0(L))$ be the operator such that $\widetilde{P}(f)=P(f)$ for all $f\in C_0(L)$, and let $i\colon \widetilde{P}(C_0(L)) \longrightarrow C_0(L)$ be the natural inclusion. Let $\eta$ be the mapping with which all $\ell_\infty^n$ spaces has the BPBp-nu for compact operators (see Corollary \ref{Corc0}). Since $n(C_0(L))=1$ and $n(\ell_p^n)=1$ for all $n\in \mathbb{N}$ (see \cite[Proposition 1.11]{KLMM} for instance), in particular, $n_k(P(C_0(L)))=n_k(C_0(L))=1$. Therefore, we are in the conditions to apply Lemma \ref{LemaBPBpnuK1} and get that $C_0(L)$ has the BPBp-nu for compact operators, as desired.

Alternatively, by Remark~\ref{remark:pi-property}, we may prove Theorem~\ref{TheoremMain} applying Proposition~\ref{Prop:Projec} instead of Lemma \ref{LemaBPBpnuK1}.
\end{proof}

\vspace*{0.5cm}

\noindent \textbf{Acknowledgment.} The authors would like to thank Bill Johnson for kindly answering several inquiries.

\end{document}